\documentclass[12pt,a4paper]{article}

\usepackage[margin=1in]{geometry}

\usepackage{enumitem}

\usepackage{eucal}

\usepackage[all]{xy}

\usepackage{upgreek}

\usepackage[utf8]{inputenc}

\usepackage{amssymb, amsmath, amsthm}

\usepackage{mathptmx}
\usepackage[scaled=.90]{helvet}
\usepackage{courier}

\usepackage[pdftex]{hyperref}

\hypersetup{
  colorlinks = true,
  linkcolor = black,
  urlcolor = blue,
  citecolor = black,
}

\renewcommand{\alpha}{\upalpha}
\renewcommand{\beta}{\upbeta}
\renewcommand{\theta}{\uptheta}
\renewcommand{\gamma}{\upgamma}
\renewcommand{\sigma}{\upsigma}
\renewcommand{\delta}{\updelta}
\renewcommand{\epsilon}{\upepsilon}
\renewcommand{\zeta}{\upzeta}
\renewcommand{\eta}{\upeta}
\renewcommand{\kappa}{\upkappa}
\renewcommand{\lambda}{\uplambda}
\renewcommand{\mu}{\upmu}
\renewcommand{\nu}{\upnu}
\renewcommand{\xi}{\upxi}
\renewcommand{\pi}{\uppi}
\renewcommand{\rho}{\uprho}
\renewcommand{\upsilon}{\upupsilon}
\renewcommand{\phi}{\upphi}
\renewcommand{\chi}{\upchi}
\renewcommand{\psi}{\uppsi}
\renewcommand{\omega}{\upomega}

\renewcommand{\varepsilon}{\upvarepsilon}
\renewcommand{\vartheta}{\upvartheta}
\renewcommand{\varpi}{\upvarpi}

\renewcommand{\phi}{\upvarphi}

\renewcommand{\varpi}{\upvarpi}
\renewcommand{\varsigma}{\upvarsigma}
\renewcommand{\varrho}{\upvarrho}

\newcommand{\ot}{\otimes}
\newcommand{\ra}{\to}

\newcommand{\rmd}{\mathrm{d}}
\newcommand{\rmi}{\mathrm{i}}

\newcommand{\bbF}{\mathbb{F}}
\newcommand{\bbZ}{\mathbb{Z}}

\newcommand{\rmC}{\mathrm{C}}
\newcommand{\rmD}{\mathrm{D}}
\newcommand{\rmH}{\mathrm{H}}
\newcommand{\rmK}{\mathrm{K}}
\newcommand{\rmQ}{\mathrm{Q}}
\newcommand{\rmR}{\mathrm{R}}
\newcommand{\rmZ}{\mathrm{Z}}

\newcommand{\frK}{\mathbb{K}}

\newcommand{\frS}{\Sigma}

\newcommand{\Z}{\mathbb{Z}}
\newcommand{\Q}{\mathbb{Q}}
\newcommand{\C}{\mathcal{C}}
\newcommand{\D}{\mathcal{D}}

\DeclareMathOperator{\id}{id}
\DeclareMathOperator{\tr}{tr}

\DeclareMathOperator{\Hom}{Hom}
\DeclareMathOperator{\End}{End}
\DeclareMathOperator{\Aut}{Aut}
\DeclareMathOperator{\Out}{Out}
\DeclareMathOperator{\BrPic}{BrPic}

\newcommand{\Rep}{\mathbf{Rep}}

\newtheorem{theorem}{Theorem}
\newtheorem{proposition}[theorem]{Proposition}

\newtheorem{lemma}[theorem]{Lemma}
\newtheorem*{theorem*}{Theorem}
\newtheorem*{question*}{Question}

\theoremstyle{definition}
\newtheorem{definition}[theorem]{Definition}
\newtheorem{example}[theorem]{Example}
\newtheorem{question}[theorem]{Question}

\newtheorem{remark}[theorem]{Remark}

\numberwithin{theorem}{section}
\numberwithin{equation}{section}

\title{Adams operations and symmetries \hbox{of representation categories}}

\author{Ehud Meir and Markus Szymik}

\date{May 2019}

\begin{document}

\maketitle

\renewcommand{\abstractname}{\vspace{-2\baselineskip}}

\begin{abstract}%
\noindent
Abstract: 
Adams operations are the natural transformations of the representation ring functor on the category of finite groups, and they are one way to describe the usual~$\lambda$--ring structure on these rings. 
From the representation-theoretical point of view, they codify some of the symmetric monoidal structure of the representation category. We show that the monoidal structure on the category alone, regardless of the particular symmetry, determines all the odd Adams operations. On the other hand, we give examples to show that monoidal equivalences do not have to preserve the second Adams operations and to show that monoidal equivalences that preserve the second Adams operations do not have to be symmetric. Along the way, we classify all possible symmetries and all monoidal autoequivalences of representation categories of finite groups. 

\vspace{\baselineskip}

\vspace{\baselineskip}
\end{abstract}


\section{Introduction}

Every finite group~$G$ can be reconstructed from the category~$\Rep(G)$ of its finite-dimensional representations if one considers this category as a symmetric monoidal category. This follows from more general results of Deligne~\cite[Prop.~2.8]{Deligne+Milne},~\cite{Deligne}. If one considers the representation category~$\Rep(G)$ as a monoidal category alone, without its canonical symmetry, then it does not determine the group~$G$. See Davydov~\cite{Davydov} and Etingof--Gelaki~\cite{Etingof+Gelaki} for such {\em isocategorical} groups. Examples go back to Fischer~\cite{Fischer}.


The representation ring~$\rmR(G)$ of a finite group~$G$ is a~$\lambda$--ring. This structure derives from the symmetric monoidal structure of the category~$\Rep(G)$. For rings which are torsion-free as abelian groups, such as~$\rmR(G)$, the structure of a~$\lambda$--ring is equivalent to the lesser-known structures of a~$\tau$--ring or $\Psi$--ring~(see~\cite{Hoffman} and~\cite{Wilkerson}). One way or another, it is determined by a family of commuting Frobenius lifts~$\Psi^p$, the \emph{Adams operations}, which evaluate characters on~$p$--th powers. The~$\lambda$--ring structure of~$\rmR(G)$ gives us more information about the group~$G$ than just the ring~$\rmR(G)$. 

For example, the representations rings of the dihedral group~$\rmD_8$ and the quaternion group~$\rmQ_8$ of order~$8$ are isomorphic as rings, but not as~$\lambda$--rings: the second Adams operation~$\Psi^2$ can be used to tell them apart. However, there are also examples~(due to Dade~\cite{Dade}) of non-isomorphic finite groups~(of order~$5^7=78125$ and more) with isomorphic representation rings as~$\lambda$--rings.~(These examples are particularly simple for the Adams operations because these groups are~$p$--groups of exponent~$p$, so that the~$p$--th Adams operation~$\Psi^p$ gives the dimension function, and the other Adams operations~$\Psi^\ell$ for~$\ell$ prime to~$p$ are given by Galois actions, regardless of the group. Dade distinguishes these groups using their Lie rings.) 

\subsection{The main result and examples}

The main result of this paper is the following.

\begin{theorem}\label{oddoperations}
If~$\Rep(G)\to\Rep(G')$ is a monoidal equivalence between representation categories of finite groups~$G$ and~$G'$, then the induced isomorphism~\hbox{$\rmR(G)\cong\rmR(G')$} between their representation rings preserves the Adams operations~$\Psi^p$ for all odd~$p$.
\end{theorem}

Of course, when the monoidal equivalence is {\it symmetric}, the consequence holds trivially for {\it all} Adams operations, since these are defined using the symmetry of the category. The point here is that we do {\em not} assume this. The result shows that the odd Adams operations do not carry additional information about the symmetry on the monoidal representation category. 

In Example~\ref{ex:2}, we show that this result cannot be improved: there exist groups~$G$ and~$G'$ together with a monoidal equivalence~$\Rep(G)\to\Rep(G')$ that does not preserve the second Adams operations~$\Psi^2$. Of course, such a monoidal equivalence cannot be symmetric.

One might then wonder if the additional preservation of the second Adams operation~$\Psi^2$ is enough to ensure that we have a symmetric monoidal equivalence. We show in Example~\ref{ex:Adams_but_not_symmetric} that this is also not the case, completing the story. 

\subsection{Implications}

Consider two finite groups~$G$ and~$G'$ and a functor~$F\colon \Rep(G)\to \Rep(G')$ which is an equivalence between the underlying~$\frK$--linear categories. We say that~$F$ preserves an Adams operation~$\Psi^n$ if the induced map~$\rmR(G)\to \rmR(G')$ preserves this operator. 
Consider the following statements about a functor~$F$.
\begin{enumerate}[noitemsep,topsep=0pt,leftmargin=5em]
\item[(F1)] The functor~$F$ is symmetric monoidal.
\item[(F2)] The functor~$F$ is monoidal and preserves all Adams operations.
\item[(F3)] The functor~$F$ is monoidal and preserves the odd Adams operations.  
\item[(F4)] The functor~$F$ is monoidal.
\end{enumerate}
Then, obviously, we have 
\[
\xymatrix{
\text{(F1)}\ar@{=>}[r]&\text{(F2)}\ar@{=>}[r]&\text{(F3)}\ar@{=>}[r]&\text{(F4)}.
}
\]
Now, as a consequence of our results, we can decide the validity of all other implications:
(F3) and~(F4) are actually equivalent by Theorem~\ref{oddoperations}, and the other implications cannot be reversed, by Examples~\ref{ex:2} and~\ref{ex:Adams_but_not_symmetric}.
\[
\xymatrix{
\text{(F1)}\ar@{<=}[r]|\not&\text{(F2)}\ar@{<=}[r]|\not&\text{(F3)}\ar@{<=}[r]&\text{(F4)}
}
\]


\parbox{\linewidth}{The examples that we have in Section \ref{sec:examples} lead to more general questions, about the mere existence of functors. Consider the following statements about two finite groups~$G$ and~$G'$.

\begin{enumerate}[noitemsep,topsep=0pt,leftmargin=5em]
\item[(G1)] The groups~$G$ and~$G'$ are isomorphic.
\item[(G2)] There exists a symmetric monoidal equivalence~$\Rep(G)\simeq\Rep(G')$. 
\item[(G3)] There exists a monoidal equivalence~$\Rep(G)\simeq\Rep(G')$ which preserves all Adams operations.
\item[(G4)] There exists a monoidal equivalence~$\Rep(G)\simeq\Rep(G')$ which preserves the odd Adams operations.
\item[(G5)] There exists a monoidal equivalence~$\Rep(G)\simeq\Rep(G')$.
\end{enumerate}}

Then, obviously, we have $\text{(G1)}\Rightarrow\text{(G2)}\Rightarrow\text{(G3)}\Rightarrow\text{(G4)}\Rightarrow\text{(G5)}$. The equivalence of~(G1) and~(G2) follows from Deligne's work. The equivalence of~(G4) and~(G5) follows from our Theorem~\ref{oddoperations}. The examples of Etingof and Gelaki show that~$\text{(G5)}\nRightarrow \text{(G1)}$, and we now know that this entails~\hbox{$\text{(G4)}\nRightarrow\text{(G2)}$}. On the other hand, we do not know if the condition~(G4) is strictly stronger than~(G3), or if~(G3) is strictly stronger than~(G2). Of course, it is not possible that~(G3) will be equivalent to both~(G2) and~(G4), because this would contradict~\hbox{$\text{(G4)}\nRightarrow\text{(G2)}$}.
\[
\xymatrix@C=1em{
\text{(G1)}\ar@{<=>}[rr]&&\text{(G2)}\ar@{=>}[dr]\ar@{<=}[rr]|\not& &\text{(G4)}\ar@{<=>}[rr]&&\text{(G5)}\\
&&&\text{(G3)}\ar@{=>}[ur]
}
\]


\subsection{Outline}

We review the Adams operations in Section~\ref{sec:lambda} and give End's characterization of these operations as the natural transformations of the representation ring functor, and offer a representation theoretic description in Proposition~\ref{prop:formula}. This result nicely complements the usual formula on the level of characters. In Section~\ref{isocats}, we review the Etingof--Gelaki classification of isocategorical groups. Building on that, we give a new description of all possible symmetries on the monoidal representation categories~$\Rep(G)$ in Theorem~\ref{thm41} and the monoidal autoequivalences of~$\Rep(G)$ in Theorem~\ref{thm42}. Section~\ref{sec:proof} contains a proof of Theorem~\ref{oddoperations}, and we give the promised examples in the final Section~\ref{sec:examples}.


\subsection{Notation and conventions}

For any finite group~$G$, we will denote by~$\Rep(\frK G)$ the category of finite-dimensional~$G$--representations over an algebraically closed field~$\frK$ of characteristic~$0$. This is more precise than the notation~$\Rep(G)$ that we have used for the purposes of the introduction only. More generally, we will use the notation~$\Rep(H)$ for the category of finite-dimensional representations of a finite-dimensional Hopf algebra~$H$ over~$\frK$. These representation categories have a canonical structure of a rigid monoidal category.~(Rigidity means that each object admits a left and right dual object. See~\cite{Kassel},~\cite{Calaque+Etingof}, ~\cite{ENO} or~\cite{EGNO} for expositions.) In addition, these representation categories admit a canonical symmetric monoidal structure: for every two objects~$V,W$ in~$\Rep(\frK G)$ we have a natural isomorphism~$\sigma_{V,W}\colon V\ot W\to W\ot V$ given by~$v\ot w\mapsto w\ot v$. This isomorphism satisfies~\hbox{$\sigma_{V,W}\sigma_{W,V} = \id$}. Our analysis of the Adams operations later will lead us to study other symmetric monoidal structures on~$\Rep(\frK G)$ as well. Notice that in the work \cite{Etingof+Gelaki} all possible pairs of isocategorical groups were classified, but the corresponding monoidal equivalences between the corresponding monoidal categories were not described. They are described here in Section \ref{isocats} and in the preprint \cite{Goyvaerts+Meir}, where it is also shown that the representation categories of isocategorical groups have isomorphic Witt groups.

Throughout the paper, we will denote by~$\rmR(G)$ the Grothendieck ring~$\rmK_0\Rep(\frK G)$ of the monoidal category~$\Rep(\frK G)$. This is the representation ring of the group~$G$, and~(by our assumptions on the field~$\frK$) independent of our particular choice of~$\frK$, as the notation suggests.
The representation ring has a canonical~$\Z$--basis given by the isomorphism classes~$[V_1],\ldots,[V_r]$ of a representative set of irreducible representations~$V_1,\ldots,V_r$ of~$G$. It can be identified with a subring of the ring of class functions on~$G$ by means of the character map that sends a representation~$V$ to its character~$\chi_V\colon g\mapsto\tr(\,g\,|\,V\,)$.
In the following, we will identify a class~$[V_i]$ with its character~$\chi_i$ whenever it will facilitate the text.


\section{Representation rings and power operations}\label{sec:lambda}

In this section we review some background material. We describe the Adams operations as the natural transformations of the representation ring functor on the category of finite groups. They are an equivalent way of encoding the well-known~$\lambda$--ring structure. We also give an interpretation in terms of the representation category, and explain how to determine the order and the exponent of any finite group from these operations.


\subsection{Representation rings as \texorpdfstring{$\lambda$}{lambda}--rings}

The notion of a~$\lambda$--ring first arose in Grothendieck's work~\cite[4.2]{Grothendieck} on vector bundles and~K-theory. 
Nowadays, it is recognized that~$\lambda$--operations are useful additional structure that is present in many other contexts. For representation theory, see~\cite{Atiyah+Tall},~\cite{Kervaire},~\cite{Kratzer}, and~\cite{Benson}. There are many ways to present the algebraic theory of~$\lambda$--rings. 
See Berthelot's chapter~\cite{Berthelot} in~SGA~6 as well as~\cite{Knutson},~\cite{Patras},~and the recent book~\cite{Yau} by Yau. 
For the purposes of the present text, it is only important to know about Wilkerson's criterion~\cite{Wilkerson}: if a ring is torsion-free as an abelian group, then a~$\lambda$--ring structure on it is equivalent to a family~$(\,\Psi^n\,|\,n\in\bbZ\,)$ of commuting ring endomorphisms that satisfy~$\Psi^{mn}=\Psi^m\Psi^n$ and such that~$\Psi^p$ is a Frobenius lift for each prime number~$p$. In particular, since representation rings are torsion-free, their~$\lambda$--ring structure is determined by the Adams operations and vice versa. 

For a natural number~$k$ and a representation~$V$ of a finite group~$G$, the~$\lambda$--operations on the representation ring~$\rmR(G)$ are defined using the exterior powers: we have~\hbox{$\lambda^k[V] = [\Lambda^k V]$}. Notice that this definition relies on the usual symmetric monoidal structure on the representation category~$\Rep(\frK G)$. The corresponding Adams operations are defined on the level of class functions by the formula~\hbox{$\Psi^k_G(\chi)(g) = \chi(g^k)$}.~(Of course, formulas like this go back to Frobenius, see~\cite[\S 8]{Frobenius}, for instance.) In case the group~$G$ is clear from the context, we will write~$\Psi^k_G=\Psi^k$. Notice that for every~$k$ and every character~$\chi$ it holds that~$\Psi^k(\chi)$ is a~$\Z$--linear combination of the irreducible characters. This is not immediately clear from the given definition of the Adams operations. It follows from the fact that they correspond to the aforementioned~$\lambda$--operations. See~\cite[Ex.~9.3.b]{Serre} and Remark~\ref{rmk:integervalues}.

\begin{remark}
Adams operations in K-theory have been introduced in~\cite[Sec.~4]{Adams}. See also the exposition~\cite{Atiyah} by Atiyah. The latter approach was axiomatized and applied to representation rings of finite groups~$G$ by Hoffman~\cite{Hoffman}, a former student of Adams. The idea is, rather than working with the exterior powers~$\Lambda^k V$ of a~$G$--representation~$V$, to use the tensor powers~$V^{\otimes k}$ as representations of the wreath products~$G\wr\Sigma_k$. The resulting theory of~$\tau$--rings turns out to be equivalent to the theory of~$\lambda$--rings~\cite[Thm.~6.6]{Hoffman}. We will not make use of these facts in the following.
\end{remark}


\subsection{A characterization of the Adams operations}

The~$k$--th Adams operation~$\Psi^k_G$ depends only the residue class of the integer~$k$ modulo the exponent of~$G$, and therefore can be defined for all elements~$k$ in the pro-finite completion~$\widehat\bbZ$ of the group~$\bbZ$ of integers.

The family~$(\,\Psi_G^k\,|\,G\,)$ defines a natural endomorphism~$\Psi^k\colon\rmR\to\rmR$ of the representation ring functor~$\rmR$, 
thought of as a contravariant functor from the category of finite groups and~(all) homomorphisms to the category of~(commutative) rings~(with unit): 
restriction along a group homomorphism~$G\to G'$ defines a ring homomorphism~$\rmR(G')\to\rmR(G)$. In fact, it is easy to see that this leads to a characterization of the Adams operations. This was proved by End~\cite{End} after the~K-theory case had been proved by tom Dieck~\cite{tomDieck}:

\begin{proposition}\label{prop:char}
{\upshape(End~\cite[Satz~(8.4)]{End})} The monoid of endomorphisms of the representation ring functor is isomorphic to the multiplicative monoid of pro-finite integers~$\widehat\bbZ$. The pro-finite integer~$k$ corresponds to the~$k$--th Adams operation~$\Psi^k$. 
\end{proposition}

The idea for a proof is very simple: the group~$\bbZ$ of integers under addition generates the category of groups~(in the sense that it~(co-)represents the forgetful functor to sets) and we could detect every natural transformation on this single example. But the group~$\bbZ$ of integers is not finite, and we have to replace it by the family of its finite quotients, the finite cyclic groups.

\begin{proof}
Let us first see that that every natural transformation~$\Phi$ has the form~$\Psi^k$ for some pro-finite integer~$k$. Since the restrictions to cyclic subgroups induce an injection
\[
\rmR(G)
\longrightarrow
\bigoplus_{%
	\substack{%
		{C\leqslant G}\\{\text{cyclic}}
		}
		}
\rmR(C),
\]
the natural transformation~$\Phi$ is determined by its behavior on cyclic groups. 
For a cyclic group~$C$ of order~$n$, the representation ring~$\rmR(C)$ is isomorphic to the group ring~$\bbZ[\,C^\vee\,]$ of the character group~$C^\vee$. 
This is generated~(as a ring) by any embedding~\hbox{$\gamma\colon C\to\frK^\times$} of groups~(and it holds that~$\gamma^n=1$). 
Then~$\Phi(\gamma)$ is also a unit that is torsion. 
By Higman's theorem~\cite[Thm.~3]{Higman}, for instance, every unit of~$\bbZ[\,C^\vee\,]$ that is torsion is trivial in Whitehead's sense: we have~$\Phi(\gamma)=\pm\gamma^{k_n}$ for some integer~$k_n$. 
The restriction to the trivial subgroup shows that~$\Phi$ preserves the augmentation~\hbox{$\rmR(G)\to\bbZ$} of the representation ring~$\rmR(G)$ given by the dimension, 
so that the sign has to be positive, and~$\Phi_C=\Psi_C^{k_n}$ for all cyclic groups~$C$ of order~$n$. Restriction to subgroups shows that~$k_n$ is congruent to~$k_m$ modulo~$m$ whenever~$m$ divides~$n$, 
so that the family~$(\,k_n\in\bbZ/n\,|\,n\geqslant1)$ describes a pro-finite integer~$k$ such that~$\Phi_C=\Psi_C^k$ for all cyclic groups~$C$. As mentioned at the beginning, 
the equation~$\Phi_G=\Psi_G^k$ for all groups~$G$ follows. Conversely, this argument also shows how to read off~$k$ from~$\Psi^k$.
\end{proof}

\begin{remark}
The reduction to cyclic groups is omnipresent in the representation theory of finite groups, because representations are determined by their characters, which are functions on the group, and every group element lies in a cyclic group. In the other direction, we have Artin's theorem, which says that each character of a finite group is a rational linear combination of characters of representations induced from cyclic subgroups. See the work of Hopkins, Kuhn, and Ravenel~\cite{Hopkins,Kuhn,Hopkins+Kuhn+Ravenel:1,Hopkins+Kuhn+Ravenel:2} for generalizations.
\end{remark}


\subsection{A representation theoretic interpretation}\label{sec:FS}

We will now explain how to find an expression for the representative matrix of the~$k$--th Adams operation~$\Psi^k\colon\rmR(G)\to\rmR(G)$ on the representation ring~$\rmR(G)$ of a finite group~$G$ in terms of the basis given by irreducible characters. 
To do so we write as before~\hbox{$\chi_1,\dots,\chi_r$} for the irreducible characters of the group~$G$, and write the class function~$\Psi^k(\chi_i)\colon g\mapsto \chi_i(g^k)$ as a linear combination
\[
\Psi^k(\chi_i) = \sum_j \Psi^k_{i,j}\chi_j
\]
of the irreducible characters. 
This gives the matrix~$(\,\Psi^k_{i,j}\,|\,1\leqslant i,j\leqslant r\,)$ of~$\Psi^k$ with respect to the chosen basis---at least in theory. The scalars~$\Psi^k_{i,j}$ can be computed directly, using representation theory, as follows.

We have the standard inner product
\begin{equation}\label{eq:inner_product}
\langle\,\alpha\,|\,\beta\,\rangle=\frac1{|G|}\sum_{g}\overline{\alpha(g)}\beta(g)
\end{equation}
for class functions. The irreducible characters form an orthogonal basis, so that we can compute
\begin{equation}\label{eq:FS}
\Psi^k_{i,j}=\langle\,\chi_j\,|\,\Psi^k(\chi_i)\,\rangle=\frac{1}{|G|}\sum_{g} \chi_j(g^{-1})\chi_i(g^k).
\end{equation}

Let~$V_i$ be an irreducible representation that corresponds to the irreducible character~$\chi_i$. We choose the indexing so that~$\chi_1$ is the trivial character, and~$V_1$ is the trivial representation. 

For a given integer~$k$, the~$k$--th Frobenius--Schur indicator of a~$G$--representation~$V$ is the trace of the~$G$--linear endomorphism
\[
\frac1{|G|}\sum_{g}g^k\colon V\longrightarrow V.
\]
Equation~\eqref{eq:FS} shows that~$\Psi^k_{1,j}$ is the~$k$--th Frobenius--Schur indicator of~$V_j$. 

Let~$\sigma_k$ be the linear operator on the~$k$--th tensor power~$V_i^{\ot k}$ that cyclically permutes the factors. 
Since this is~$\frK G$--linear, we have an induced map~$(\sigma_k)_*$ on the space~$\Hom_G(V_j,V_i^{\ot k})$ by post-composition.
In order to describe~$\Psi^k_{i,j}$ explicitly, we need the following lemma, whose proof can be found in~\cite[Sec.~2.3]{KSZ}:

\begin{lemma}\label{linearalgebra}
Given any~$\frK$--vector space~$V$, let~\hbox{$\sigma_k\colon V^{\otimes k}\to V^{\otimes k}$} denote the cyclic permutation. Given~$\frK$--linear endomorphisms~\hbox{$f_1,\ldots, f_k\colon V\to V$}, we have an equality
\[
\tr(\,\sigma_k(f_1\ot f_2\ot\cdots \ot f_k)\,|\,V^{\otimes k}\,) 
= 
\tr(\,f_1 f_2 \cdots f_k\,|\,V\,)
\]
of traces.
\end{lemma}

\parbox{\linewidth}{\begin{proposition}\label{prop:formula}
Let~$V_1,\dots,V_r$ be representatives of the isomorphism classes of irreducible~$G$--representations of a finite group~$G$. Then we have
\[
\Psi^k(\chi_i) = \sum_{j=1}^r \tr\left(\,\sigma_k^*\,|\,\Hom_G(V_j,V_i^{\ot k})\,\right) \chi_j
\]
where~$V_i$ is the irreducible representation which corresponds to the character~$\chi_i$ in the representation ring~$\rmR(G)$, where~$\sigma_k$ is the cyclic permutation of the tensor factors of~$V_i^{\ot k}$.
\end{proposition}
}

See Prop.~2.5 in~\cite{Atiyah} and the remarks following it for a similar statement when~$k=p$ is a prime~(and in the context of vector bundles).

\begin{proof}
Consider the vector space~$W=\Hom_G(V_j,V_i^{\ot k})$.
Notice first that~$W$ is the subspace of~$G$--invariants in the representation~$U=\Hom(V_j,V_i^{\ot k})$ where the action of the group~$G$ is the usual diagonal action.
Also, the~$G$--representation~$U$ is naturally isomorphic with~\hbox{$V_j^*\ot V_i^{\ot k}$}. The projection~$U\ra W$ is given by the action of the idempotent~$\epsilon=|G|^{-1}\sum_{g\in G}g$. 
The operator~$\sigma^*$ commutes with the action of~$\epsilon$, and we have~\hbox{$\tr(\,\sigma^*\,|\,W\,)=\tr(\,\epsilon\sigma^*\,|\,U\,).$} This is in fact true whenever we have an operator commuting with a projection. As in Lemma~\ref{linearalgebra} we compute 
\begin{align*}
\tr(\,\sigma^*\,|\,W\,) 
&=\tr(\,\epsilon\sigma^*\,|\,U\,)\\ 
&=\tr\left(\,\frac{1}{|G|}\sum_{g\in G} g\ot g^k\,|\,V_j^*\ot V_i\,\right)\\
&=\frac{1}{|G|}\sum_{g\in G}\chi_j(g^{-1})\chi_i(g^k)\\
&=\Psi^k_{i,j}
\end{align*}
to finish the proof.
\end{proof}

\begin{remark}\label{rmk:integervalues}
We have shown that~$\Psi^k_{i,j}$ is the character value of a~$k$--cycle for some representation of the permutation group~$\frS_k$. This makes it also clear that the coefficients~$\Psi^k_{i,j}$ are integers: all irreducible representations of~$\frS_k$ are already realizable over the prime field~$\Q$, and character values are always algebraic and contained in the field of definition.
\end{remark}


\subsection{Order and exponent}

To end this section, we explain how to extract some basic numerical information about a group from its representation ring~(as a~$\lambda$--ring). The results here will not be used in the rest of the paper.

Recall that the rank of the representation ring~$\rmR(G)$~(as an abelian group) is the number of conjugacy classes of elements in~$G$, so that this numerical invariant of~$G$ is determined by~$\rmR(G)$. We include here a proof that~$\rmR(G)$ determines also the order of~$G$, as we are unaware of an appropriate reference for that.

\begin{proposition}
The representation ring~$\rmR(G)$ determines the order of~$G$.
\end{proposition}

\begin{proof}
We can identify~$\rmR(G)$ with a subring of the ring~$\mathbb{C}\ot_{\Z} \rmR(G)$ and think of this as the ring of complex class functions on~$G$.
The primitive central idempotents in~$\mathbb{C}\ot_{\Z} \rmR(G)$ are the characteristic functions of the conjugacy classes. Let us write~$\epsilon_{g}$ for the one corresponding to the conjugacy class of the element~$g$. The inner product~\eqref{eq:inner_product} evaluates to
\[
\langle\,\chi_j\,|\,\epsilon_g\,\rangle=\frac1{|G|}\sum_h\chi_j(h^{-1})\epsilon_g(h),
\]
which is~$\chi_j(g^{-1})/|G|$ times the number of conjugates of~$g$ in~$G$, or~$|G/\rmC_G(s)|$. 
Since the irreducible characters form an orthonormal basis, we get the formula
\[
\epsilon_g = \frac1{|\rmC_G(g)|}\sum_j\chi_j(g^{-1})\chi_j.
\]
For instance, we have~$\epsilon_1=\rho/|G|$. In general, the coefficient of the trivial representation is~$1/|\rmC_G(g)|$. Let~$n_g$ be the minimal positive integer such that~$n_g\epsilon_g$ is contained in~$\rmR(G)$ or~$0$ if such an integer does not exist. Then the maximum~$\max\{\,n_g\,|\,g\in G\,\}$ is achieved at the identity element, and it is~$n_e=|G|$. In this way we recover~$|G|$ just from the ring~$\rmR(G)$.
\end{proof}

A little more work, and the Adams operations, give the exponent:

\begin{proposition}\label{prop:exponent}
The~$\lambda$--ring structure on the representation ring~$\rmR(G)$ determines the exponent of~$G$.
\end{proposition}

\begin{proof}
Let~$e=\exp(G)$ be the exponent of~$G$. Then~$g^{k+e}=g^k$ for all elements~$g$ of~$G$, so that we have~$\Psi^{k+e}=\Psi^k$ for all integers~$k$. 
This shows that the family~\hbox{$\Psi=(\,\Psi^k\,|\,k\in\bbZ\,)$} is periodic, and we claim that~$e$ is its period.

We can use that~$\Psi^0=\dim$ is the dimension function~$\rmR(G)\to\bbZ\to\rmR(G)$. Assume that we have~$\Psi^{k+f}=\Psi^k$ for all integers~$k$. 
If~$\rho$ denotes the character of the regular representation, 
then~$\Psi^f(\rho)=\Psi^0(\rho)=\dim(\rho)$ is constant and equals the order of~$G$. But, on the other hand, we have~$(\Psi^f\rho)(g)=\rho(g^f)$. 
Since the character of the regular representation vanishes away from the neutral element of~$G$, we deduce~$g^f$ is the neutral element of the group~$G$ for all~$g$ so that~$f$ is a multiple of the exponent of~$G$, as claimed.
\end{proof}

\begin{remark}
It follows from Shimizu's work~\cite{Shimizu} that isocategorical groups have the same exponent. This shows that the exponent is an invariant which is also computable from the monoidal category~$\Rep(KG)$ alone, not using its symmetry. It is arguably much easier to use the~$\lambda$--ring structure on~$\rmR(G)$, though.
\end{remark}


\section{Isocategorical groups and symmetries}\label{isocats}

We have already mentioned in the course of the introduction that the representation category~$\Rep(\frK G)$ of a finite group~$G$, as a symmetric monoidal category, determines the group~$G$ up to isomorphism. Let us be more precise now. 

Deligne~\cite{Deligne} showed that every symmetric monoidal category~$\C$ that satisfies certain finiteness conditions admits a unique symmetric fiber functor~$F\colon \C\to \Rep(\frK)$. (See also Breen's exposition in~\cite{Breen}.) If~$G$ is the group of monoidal autoequivalences of the functor~$F$, then Tannaka reconstruction gives us an equivalence of symmetric monoidal categories~$\C\simeq\Rep(\frK G)$. 
Under this equivalence, the functor~$F$ gives us the forgetful functor~$\Rep(\frK G)\to\Rep(\frK)$. The symmetric monoidal structure is crucial here: it is possible that~$\Rep(\frK G)$ and~$\Rep(\frK G')$ will be equivalent as monoidal categories without the two groups~$G$ and~$G'$ being isomorphic. Stated differently, it is possible that a representation category~$\Rep(\frK G)$ admits a non-canonical symmetry for its monoidal structure which will make it equivalent, as a symmetric monoidal category, to~$\Rep(\frK G')$~(with its canonical symmetry) for a group~$G'$ that is not isomorphic to~$G$.

Two groups~$G$ and~$G'$ are called {\it isocategorical} in case~$\Rep(\frK G)$ and~$\Rep(\frK G')$ are equivalent as monoidal categories. In~\cite{Etingof+Gelaki}, Etingof and Gelaki constructed all examples of non-isomorphic isocategorical finite groups, up to isomorphism. What we add to this here is an explicit description of monoidal equivalences, which we dub {\em Etingof--Gelaki} equivalences, see Definition~\ref{def:EG} below.
 
\begin{example}
 It is shown in~\cite{Goyvaerts+Meir} that the smallest examples of isocategorical groups have order~$64$. From the~267 isomorphism classes of groups of that order, these are the two groups that are named~$\Gamma_{26} a_2$ and~$\Gamma_{26} a_3$ by Hall and Senior~\cite{Hall+Senior}, or{\tt~SmallGroup(64,136)} and{\tt~SmallGroup(64,135)} in~\cite{GAP}, respectively. They both have exponent~$8$ and representations rings of rank~$16$, but the ranks of their Burnside rings differ: they are~$76$ and~$73$, respectively. These groups can also be distinguished by looking at their homology~$\rmH_3(G\,;\,\bbZ)$ or cohomology~$\rmH^3(G\,;\,\bbZ/2)$.
\end{example}

We now review the work~\cite{Etingof+Gelaki} of Etingof and Gelaki that describes a construction of all groups, up to isomorphism, that are isocategorical to a given group~$G$

First of all, we consider a presentation
\begin{equation}\label{eq:extension}
1\longrightarrow A\longrightarrow G\longrightarrow Q\longrightarrow 1
\end{equation}
of the group~$G$ as an extension of some group~$Q$ by an abelian group~$A$. Let us agree to write~\hbox{$A^{\vee}=\Hom(A,\frK^{\times})$} for the character group of the kernel. The quotient group~$Q$ acts on the groups~$A$ and~$A^\vee$, and we can form the split extension
\[
1\longrightarrow A^{\vee}\longrightarrow A^{\vee}\rtimes Q\longrightarrow Q\longrightarrow 1.
\] 
The first differential on the second page of the Lyndon--Hochschild--Serre spectral sequence for that extension is a homomorphism
\[
\rmd_2\colon\rmH^0(Q\,;\,\rmH^2(A^{\vee}\,;\,\frK^{\times}))\longrightarrow\rmH^2(Q\,;\,\rmH^1(A^{\vee}\,;\,\frK^{\times}))= \rmH^2(Q\,;A),
\]
and we need to make this homomorphism explicit.

Let
\[
\alpha\colon A^{\vee}\times A^{\vee}\longrightarrow \frK^{\times}
\]
be a non-degenerate~$2$--cocycle for the character group such that the cohomology class~$[\alpha]$ of the~$2$--cocycle~$\alpha$ is~$Q$--invariant. In formulas, this means~\hbox{$[q(\alpha)]=[\alpha]$} for all elements~$q$ in the quotient group~$Q$. Therefore, we can choose, for each element~$q$ in the quotient group~$Q$, a~$1$--cochain~$z(q)\colon A^{\vee}\to \frK^\times$ such that
\[
\rmd z(q) = \frac{q(\alpha)}{\alpha}
\]
in the standard cochain complex~$\rmC^1(A^{\vee}\,;\,\frK^{\times})\to\rmC^2(A^{\vee}\,;\,\frK^{\times})$. We will work with one such chosen family~$(\,z(q)\,|\,q\in Q\,)$ from now on. We define, for all elements~$p,q$ in~$Q$,
\begin{equation}\label{b}
b(p,q) = \frac{z(pq)}{z(p)p(z(q))}\colon A^{\vee}\to \frK^\times.
\end{equation}
Such a function is a~$1$--cochain, and a direct calculation shows that the differential vanishes on it:~$\rmd(b(p,q))(\phi)=1$ for all~$\phi\in A^{\vee}$, so that~$b(p,q)$ is a~$1$--cocycle, that is a~$1$--dimensional representation~\hbox{$A^{\vee}\to\frK^\times$} of~$A^{\vee}$. 
In other words, the function~$b$ assigns to any pair of elements in~$Q$ an element in~$(A^{\vee})^{\vee}=A$. It turns out that~$b$ is a~$2$--cocycle of~$Q$, and we can think of it as having values in~$A$, that is~$b\in\rmZ^2(Q\,;\,A)$. The differential is given by~$\rmd_2[\alpha]=[b]$, see~\cite[Appendix]{ENOM}, for instance.

Now that we have described a cocycle~$b\in\rmZ^2(Q\,;\,A)$, we can use it to define an isocategorical group~$G_b$. It has the same underlying set as the group~$G$, but its multiplication~$\cdot_b$ is defined by the rule
\begin{equation}\label{eq:G_b}
g\cdot_b h = b(\bar{g},\bar{h})gh,
\end{equation} 
where~$g\mapsto\bar{g}$ is the projection~$G\to Q$.

\begin{theorem}{\upshape(Etingof--Gelaki~\cite[Theorem~1.3]{Etingof+Gelaki})}\label{thm:EG}
Two finite groups~$G$ and~$G'$ are isocategorical if and only if the group~$G'$ is isomorphic to a group of the form~$G_b$ for some~$b$, where the group~$A$ has order~$2^{2m}$ for some~$m$.
\end{theorem}

Etingof and Gelaki classified all pairs of isocategorical groups, but they did not describe explicitly the categorical equivalences. We do this here; we describe explicitly the monoidal equivalence 
\begin{equation}\label{F_b}
F_b\colon \Rep(\frK G)\ra\Rep(\frK G_b).
\end{equation}
arising from their construction. 
For this, the normal subgroup~$A$ does not have to be a $2$--group, and the cocycle~$\alpha$ does not have to be non-degenerate.
However, we will assume that~$\alpha$ is non-degenerate, since if~$\alpha$ is degenerate we can always reduce to a subgroup~$B$ of~$A$ that will admit such a non-degenerate~$2$--cocycle. Without further ado, we describe~\eqref{F_b}. If~$V$ is a~$G$--representation, then we set~$F_b(V)=V$, the same underlying vector space, but with the action of an element~$g$ in~$G_b$ given as follows: if
\begin{equation}\label{eq:decomposition}
V=\bigoplus_{\phi\in A^{\vee}} V(\phi)
\end{equation}
is the isotypical decomposition of~$V$ as an~$A$--representation, then
\begin{equation}\label{eq:action}
g\cdot_b v = \frac1{z(\bar{g})(\bar{g}(\phi))}gv
\end{equation} 
for the elements~$v\in V(\phi)$. Then~$F_b$ is an equivalence of categories. We can define a monoidal isomorphism~\hbox{$F_b(V\ot W)\ra F_b(V)\ot F_b(W)$} by
\[
v\ot w\longmapsto \frac1{\alpha(\phi_1,\phi_2)}v\ot w
\] 
where~$v\in V(\phi_1)$ and~$w\in W(\phi_2)$. This ends our description of the monoidal equivalence~$F_b$~(see also~\cite{Goyvaerts+Meir}).

\begin{definition}\label{def:EG}
We will call a functor of the form~$F_b$ an \textit{Etingof--Gelaki equivalence}.
\end{definition}

\begin{remark}\label{correspondences}
In order to have an equivalence, it is not necessary that~$A$ is a~$2$--group. If~$b'$ is the resulting cocycle from a different choice of the family~\hbox{$(\,z(q)\,|\,q\in Q\,)$}, 
then~$[b]=[b']$ and there exists a canonical isomorphism of groups~$\rho\colon G_{b'}\cong G_{b}$. The functor~$F_{b'}$ will then be isomorphic to the composition
\[
\Rep(\frK G)\stackrel{F_b}{\longrightarrow} \Rep(\frK G_b)\stackrel{\rho^*}{\longrightarrow} \Rep(\frK G_{b'}).
\]
In particular, if~$b$ is trivial~(which is always the case, for instance, when the order of~$A$ is odd), we will get in this way an autoequivalence of~$\Rep(\frK G)$. 
However, in case~$A$ is non-trivial, the resulting autoequivalence will be monoidal but not symmetric.
\end{remark}

The induced symmetry on the category~$\Rep(\frK G)$ that is obtained from conjugating the symmetry on~$\Rep(\frK G_b)$ with the equivalence~\hbox{$\Rep(\frK G)\to \Rep(\frK G_b)$} will be
\begin{equation}\label{eq:symmetry}
v\ot w \mapsto \frac{\alpha(\phi_1,\phi_2)}{\alpha(\phi_2,\phi_1)}w\ot v.
\end{equation}
Notice that even if~$b$ is a trivial cocycle we might get a new non-trivial symmetry on the category~$\Rep(\frK G)$. 

In the next section, we prove that these are all possible symmetric monoidal structures that arise from equivalences between representation categories of finite groups. We shall also classify all monoidal autoequivalences of representation categories of finite groups.


\section{Symmetries and monoidal autoequivalences}\label{sec:sym_monoidal_auto}

In order to state our results on the classification of symmetries, we will recall here first some notions and results from the theory of fusion categories.
Recall that a {\em fusion category} over~$\frK$ is a semi-simple rigid monoidal~$\frK$--linear category with finitely many simple objects and finite-dimensional Hom spaces~\cite{ENO}. 
A guiding example to keep in mind is the representation category~$\Rep(\frK G)$ of a finite group~$G$. A \textit{fiber functor} on a fusion category is a~$\frK$--linear exact faithful monoidal functor~\hbox{$T\colon\D\to\Rep(\frK)$}.~(The target category~$\Rep(\frK)$ is just the category of vector spaces over the ground field~$\frK$.) The endomorphism algebra~$\End_\frK(T)$ has a canonical structure of a Hopf algebra, and the functor~$T$ induces an equivalence of monoidal categories between~$\D$ and~$\Rep(\End_\frK(T))$. 
This is what is known as {\it Tannaka reconstruction} for Hopf algebras.

If now~$\D_1$ and~$\D_2$ are two fusion categories, and~$F\colon\D_1\to \D_2$ is a monoidal functor between these, and~$T\colon\D_2\to \Rep(\frK)$ is a~$\frK$--linear monoidal functor, 
then~$F$ induces a Hopf algebra morphism~$\End_\frK(T)\to\End_\frK(TF)$ and the diagram 
\[
\xymatrix{
\D_1 \ar[r] \ar[d] & \D_2\ar[d] \\ 
\Rep(\End(TF))\ar[r] & \Rep(\End(T))
}
\]
is commutative up to a natural equivalence, where the lower horizontal arrow is restriction of representations. 

Let us specialize these considerations to the case where~$H$ is a finite-dimensional semi-simple Hopf algebra over the field~$\frK$, and where~$T\colon\Rep(H)\to\Rep(\frK)$ is the forgetful functor. 
Then a monoidal autoequivalence~\hbox{$L\colon\Rep(H)\to \Rep(H)$} is given by a pair~$(T',\Phi)$, where the first entry~$T'\colon\Rep(H)\to\Rep(\frK)$ is a~$\frK$--linear monoidal functor 
(which will become the composition~$TL$) and the second entry is an isomorphism~\hbox{$\Phi\colon H\cong\End_\frK(T')$} of Hopf algebras. 
One can show that~$(T',\Phi)$ will define the identity autoequivalence if and only if~$T'\cong T$, and under this isomorphism it holds that~$\Phi$ corresponds to an automorphism of~$H$ which arises from conjugation by a group-like element.

Consider now the case where the Hopf algebra~$H=\frK G$ is the group algebra of a finite group~$G$. Following the works of Movshev~\cite{Movhshev}, Davydov~\cite{Davydov}, Ostrik~\cite{Ostrik1,Ostrik2}, and Natale~\cite{Natale}, we know that fiber functors~\hbox{$F\colon\Rep(\frK G)\to \Rep(\frK)$}
are in one-to-one correspondence with pairs of the form~$(A,[\psi])$ where~$A$ is a subgroup of~$G$ and~$[\psi]\in \rmH^2(A\,;\,\frK^{\times})$ is a cohomology class of a non-degenerate~$2$--cocycle.
Two pairs~$(A_1,[\psi_1])$ and~$(A_2,[\psi_2])$ define isomorphic functors if and only if they differ by conjugation in~$G$.
For instance, the pair~$(1,[1])$ consisting of the trivial subgroup and the trivial~$2$--cocycle on it corresponds to the forgetful functor~$\Rep(\frK G)\to \Rep(\frK)$.
Furthermore, the Hopf algebra that one receives from Tannaka reconstruction is again a group algebra if and only if the subgroup~$A$ is abelian and normal, 
and the cohomology class of the~$2$--cocycle~$\psi$ is invariant under the action of the group~$G$.

In the case of abelian groups the cocycle~$\psi$ amounts to a skew-symmetric pairing on~$A$, given by 
\[
(a,b)\longmapsto\frac{\psi(a,b)}{\psi(b,a)}.
\] 
In case~$\psi$ is non-degenerate, this gives us an isomorphism~$S\colon A\cong A^{\vee}$. 
The~$2$--cocycle~$(S^{-1})^*(\psi)$ is then exactly the~$2$--cocycle which appears in the Etingof--Gelaki classification of isocategorical groups mentioned in the previous section. 
This discussion leads to the following results:

\begin{theorem}\label{thm41}
Assume that~$\sigma$ is a symmetric monoidal structure on~$\Rep(\frK G)$. Assume also that the exterior powers~$\Lambda^nV$~(defined using the symmetry) satisfy~\hbox{$\Lambda^nV=0$} for all~$V$ and large enough~$n$~(depending on~$V$).
Then there exists a normal abelian subgroup~$A$ of~$G$ and a~$G/A$--invariant non-degenerate cohomology class of a~$2$--cocycle 
$\alpha\colon A^{\vee}\times A^{\vee}\to \frK^{\times}$ such that the symmetry is given by the formula~\eqref{eq:symmetry}:
\[
v\ot w\longmapsto\frac{\alpha(\phi_1,\phi_2)}{\alpha(\phi_2,\phi_1)}w\ot v,
\]
where, in the notation of~\eqref{eq:decomposition}, the vector~$v$ is in~$V(\phi_1)$ and~$w$ is in~$V(\phi_2)$.
\end{theorem}

\begin{proof}
Let us first assume that the symmetry~$\sigma$ satisfies the condition on exterior powers. From Deligne's work~\cite{Deligne}
we know that there exists an equivalence of symmetric monoidal categories~\hbox{$(\Rep(\frK G),\sigma)\to(\Rep(\frK G'),\sigma ')$} where~$\sigma '$ is the canonical symmetry on the category~$\Rep(\frK G')$ for some finite group~$G'$. 
By the above discussion and the classification of fiber functors on~$\Rep(\frK G)$ we know that every monoidal equivalence between~$\Rep(\frK G)$ and~$\Rep(\frK G')$ is an Etingof--Gelaki equivalence. 
This implies that the symmetric monoidal structure is of the claimed form.
\end{proof}

\begin{remark} 
The finiteness condition~($\Lambda^nV=0$ for all~$V$ and large enough~$n$) is necessary. Indeed, there are symmetries on monoidal categories of the form~$\Rep(\frK G)$ that do not satisfy this~(use \cite[Ex.~0.4(i)]{Deligne2}: for these symmetries, symmetric powers and exterior powers interchange, so that the latter are all non-zero). 
\end{remark}

\begin{theorem}\label{thm42}
Any monoidal autoequivalence of~$\Rep(\frK G)$ is given by a triple~$(A,\alpha,\phi)$ consisting of an abelian normal subgroup~$A$ of~$G$, a~$G/A$--invariant non-degenerate~$2$--cocycle~$\alpha$ on~$A^{\vee}$ and an isomorphism~$\phi\colon G\to G_b$, where~$b$ is the~$2$--cocycle from the Etingof--Gelaki construction. 
Moreover, two tuples~$(A,\alpha,\phi)$ and~$(A',\alpha',\phi')$ will define isomorphic equivalences if and only if~$A=A'$,~$\alpha=\alpha'$ and~$\phi$ differs from~$\phi'$ by conjugation with an element of the group~$G$.
\end{theorem}

\begin{proof}
The proof of this is similar to the one before.
\end{proof}

\begin{remark}
There is a connection between the group of monoidal autoequivalences of the fusion category~$\Rep(\frK G)$ that appears here and the Brauer--Picard group~$\BrPic(\Rep(\frK G))$ that has been introduced in~\cite{ENOM}. We have a natural homomorphism 
\[
\Phi\colon\Aut_{\ot}(\Rep(\frK G))\longrightarrow\BrPic(\Rep(\frK G))
\]
of groups which sends an automorphism~$\psi$ to the quasi-trivial bimodule~$\Rep(\frK G)_{\psi}$. 
The kernel of this homomorphism contains all autoequivalences that are given~(up to isomorphism) by conjugation with an invertible object in the category~$\Rep(\frK G)$. 
Since the category~$\Rep(\frK G)$ is symmetric, the homomorphism~$\Phi$ is in fact injective,
and we can consider the group~$\Aut_{\ot}(\Rep(\frK G))$ as a subgroup of the Brauer--Picard group~$\BrPic(\Rep(\frK G))$. The image of the homomorphism~$\Phi$ is denoted by~$\Out(\Rep(\frK G))$.
An element in the Brauer--Picard group is given by a module category~$\mathcal{M}$ over~$\Rep(\frK G)$ together with an equivalence between~$\Rep(\frK G)$ and the dual~$\Rep(\frK G)_{\mathcal{M}}$.
Equivalence classes of module categories over~$\Rep(\frK G)$ for which the dual is equivalent to~$\Rep(\frK G')$ for some group~$G'$ are given by pairs~$(A,\psi)$ where~$A$ is an abelian normal subgroup,
and~$\psi$ is a~$2$--cocycle, not necessarily non-degenerate.
We see that the image of~$\Phi$ contains all the bimodule categories in which the cocycle~$\psi$ is non-degenerate~(and for which the resulting group is isomorphic to~$G$).
For a deeper study of the Brauer--Picard group of~$\Rep(\frK G)$ we refer the reader to the paper \cite{NikRie}. 

Note also that we only describe the elements of the~(categorical) automorphism group of~$\Rep(\frK G)$.
We do not explain the composition. That is a different story, and it is the core difficulty in many calculations done for the Brauer--Picard group. See~\cite{Lentner+Priel}, which is related to that problem.
\end{remark}


\section{Preservation of the odd Adams operations}\label{sec:proof}

The purpose of this section is to prove Theorem~\ref{oddoperations}.

As we have already noted in the proof of Theorem~\ref{thm41}, the discussion in Section~\ref{sec:sym_monoidal_auto} shows that every monoidal equivalence between representation catgegories is an Etingof--Gelaki equivalence. Let us, therefore, begin with the following statement.

\begin{lemma}
Assume that we have a short exact sequence~$1\to A\to G\to Q\to 1$, a non-degenerate~$2$--cocycle~$[\alpha]\in \rmH^2(A^{\vee}\,;\,\frK^{\times})^Q$, and~$b\colon Q\times Q\to A$ the resulting~$2$--cocycle from the Etingof--Gelaki construction.
If~$b'\colon Q\times Q\to A$ is the resulting~$2$--cocycle from a different choice of family~$(\,z(q)\,|\,q\in Q\,)$, then for every integer~$k$ the functor~$F_b$ preserves~$\Psi^k$ if and only if the functor~$F_{b'}$ preserves~$\Psi^k$.
\end{lemma}

\begin{proof}
Recall from Remark \ref{correspondences}: if~$b'$ is the resulting cocycle from a different choice of the family~\hbox{$(\,z(q)\,|\,q\in Q\,)$}, 
then~$[b']=[b]$, there exists a canonical isomorphism of groups~\hbox{$\rho\colon G_{b'}\cong G_{b}$}, and the functor~$F_{b'}$ is isomorphic to the composition
\[
F_{b'}\colon\Rep(\frK G)\stackrel{F_b}{\longrightarrow} \Rep(\frK G_b)\stackrel{\rho^*}{\longrightarrow} \Rep(\frK G_{b'}).
\]
The result follows from this and the fact that the symmetric monoidal equivalence~$\rho^*$ preserves all Adams operations: the Adams operations are determined by the~$\lambda$--ring structure, 
and this structure in turn is determined by the symmetric monoidal structure.  
\end{proof}

We can~(and will) therefore assume that we are dealing with an Etingof--Gelaki equivalence between the finite groups~$G$ and~$G_b$ which arises from 
an extension~\hbox{$A\to G\to Q$}, 
a~$Q$--invariant cohomology class~$[\alpha]$ in~$\rmH^2(A^{\vee}\,;\,\frK^{\times})$, and a~$2$--cocycle~$b$ as above. The group~$A$ is a~$2$--group. 

\begin{lemma}\label{lem:z}
We can choose the functions~$z(q)\colon A^{\vee}\to \frK^\times$ so that we have~\hbox{$z(q)(\phi)=1$} whenever~$q(\phi)=\phi$.
\end{lemma}

\begin{proof}
For an element~$q\in Q$, pick any function~$z(q)$ that satisfies the equation~\hbox{$\rmd z(q) = q(\alpha)/\alpha$}. The restriction of the function~$z(q)$ to the~$q$--invariant subgroup
\[
B(q)=(A^{\vee})^{q}
\]
is then a~$1$--dimensional representation of the abelian group~$B(q)$: we calculate
\[
\rmd(z(q)|{B(q)}) = \frac{q(\alpha|{B(q)})}{\alpha|{B(q)}} = 1,
\]
by invariance. Since the field~$\frK$ is assumed to be algebraically closed, its group of units is divisible, so that any homomorphism~\hbox{$B(q)\to\frK^\times$} can be extended to~$A^{\vee}$. 
By multiplying~$z(q)$ with the inverse of an extension, we do not change the~$2$--cocycle~$\rmd z(q)$, but we assure that we have~$z(q)|{B(q)}=1$, as desired. 
\end{proof}

From now on we assume that~$z(q)(\phi)=1$ whenever~$q(\phi)=\phi$. 
Under this assumption, we can prove the following lemma:

\begin{lemma}\label{lem:previous}
The isomorphism~$\rmR(G)\cong\rmR(G_b)$ induced by~$F_b$ is the identity: if~$\chi$ is the character of a representation~$V$, and~$\chi_b$ is the character of the representation~$F_b(V)$, then~$\chi_b=\chi$.
\end{lemma}

Since~$G_b$ is equal to~$G$ as a set, this formulation makes sense. 

\begin{proof}
Let~$V$ be a representation of the group~$G$ with character~$\chi$. 
Recall our notation~\hbox{$B(q)=(A^{\vee})^q$} for elements~$q$ in the quotient~$Q=G/A$. 
When we look at the way that an element~$g$ of~$G$ permutes the isotypical summands~$V(\phi)$ of~$V$, the only places where we will get a contribution to the character are those~$V(\phi)$ which are stable under the action of~$\bar{g}$. 
Therefore, we can write 
\begin{equation}\label{eq:character}
\chi(g) = \sum_{\phi\in B(\bar{g})}\tr(\,g\,|\,{V(\phi)}\,).
\end{equation}
Let~$\chi_b$ be the character of the~$G_b$--representation~$F_b(V)$.
We then have 
\[
\chi_b(g) = \sum_{\phi\in B(\bar{g})}\frac1{z(\bar{g})(\bar{g}(\phi))}\tr(\,g\,|\,{V(\phi)}\,)
\]
by the definition~\eqref{eq:action} of the new action on~$V$ when restricted to the subspace~$V(\phi)$. Since we can assume, by Lemma~\ref{lem:z}, that we have~$z(\bar{g})(\phi)=1$ 
whenever~$\bar{g}$ fixes~$\phi$, we deduce~\hbox{$\chi=\chi_b$} as functions on~$G=G_b$ and we are done.
\end{proof}

Now that we have an explicit model for an isomorphism between the representation rings~$\rmR(G)$ and~$\rmR(G_b)$, we can check that it respects the odd Adams operations:

\begin{proof}[Proof of~Theorem~\ref{oddoperations}]
Let~$k$ be an odd integer. We have~$(\Psi^k\chi)(g)=\chi(g^k)$ by definition of the Adams operations and similarly for~$\Psi^k\chi_b$. 
We already know that~\hbox{$\chi_b=\chi$} by the previous Lemma~\ref{lem:previous}. 
Therefore, the only possible difference is in the~$k$--th power~$g^k$ of~$g$: it is calculated once in the group~$G$ and once in the group~$G_b$, respectively. We can assume that~$k$ is a positive integer. When we calculate~$g^k$ using the multiplication~\eqref{eq:G_b} in the group~$G_b$, we get
\[
b(\bar{g},\bar{g})\cdot b(\bar{g}^2,\bar{g})\cdot\ldots\cdot b(\bar{g}^{k-1},\bar{g})\cdot g^k
\]
in terms of~$G$. When we use formula~\eqref{b}, we can cancel repeating terms, and the result is 
\begin{equation}\label{eq:fraction}
\frac{z(\bar{g}^k)}{z(\bar{g})\cdot g(z(\bar{g}))\cdot\ldots\cdot g^{k-1}(z(\bar{g}))}\cdot g^k.
\end{equation}
The fractional part here, which we will denote by~$a=a_1/a_2$, is an element of~$(A^\vee)^\vee\cong A$, described as a~$1$--dimensional representation~$A^{\vee}\ra \frK^{\times}$, 
and to finish the proof it suffices to show that this element does not change the value of the character.

Using the formula~\eqref{eq:character} from the proof of the previous Lemma~\ref{lem:previous}, we know that
\[
(\Psi^k\chi)(g) = \chi(g^k) = \sum_{\phi\in B(\bar{g}^k)}\tr(\,g^k\,|\,{V(\phi)}\,),
\]
and we have a similar expression for~$\Psi^k(\chi_b)$. Therefore, we only need to consider the values~\hbox{$a(\phi)=a_1(\phi)/a_2(\phi)$} of the element~$a$ on~$\phi\in B(\bar{g}^k)$. Recall that these~$\phi\in B(\bar{g}^k)$ are homomorphisms~\hbox{$\phi\colon A\to\frK^\times$} that are fixed by~$g^k$.

By the assumption that we made~(after Lemma~\ref{lem:z} and its proof) on the functions~$z(q)$, and because~\hbox{$\phi\in A^\vee$} is fixed by~$g^k$, we get~$a_1(\phi)=z(\bar{g}^k)(\phi)=1$ for the numerator of~\eqref{eq:fraction}. 

The denominator
\begin{equation}\label{eq:Z}
a_2(\phi)=z(\bar{g})(\phi)z(\bar{g})(g(\phi))\cdots z(\bar{g})(g^{k-1}(\phi)).
\end{equation}
of~\eqref{eq:fraction} needs more effort.
We shall prove~$a_2(\phi)=1$. 

On the one hand, the order of~$a_2(\phi)$ is~$2^m$ for some~$m$. 
This follows from the fact that~$A$ is a~$2$--group and that all the values of~$\alpha$ and of~$z(q)$ can be chosen to have values which are~$2^l$--th roots of unity, for some~$l$. 
(In case the group~$A$ has odd order we can choose a representative of~$[\alpha]$ which is~$Q$--invariant as a function. It is then easy to prove that in this case all Adams operations are being preserved.)
On the other hand, we will show now that~$a_2(\phi)^k=1$. Since~$k$ is odd, this will finish the proof.

To prove~$a_2(\phi)^k=1$, we first extend the~$2$--cocycle~$\alpha$ in the following way:
if~$\phi_1,\phi_2,\ldots \phi_k$ all lie in~$A^{\vee}$, then we define
\[
\alpha(\phi_1,\phi_2,\ldots,\phi_k) = \alpha(\phi_1,\phi_2)\alpha(\phi_1\phi_2,\phi_3)\ldots\alpha(\phi_1\phi_2\cdots \phi_{k-1},\phi_k).
\]
In other words, if we think of the~$2$--cocycle~$\alpha$ as realizing an extension
\[
1\longrightarrow\frK^{\times}\longrightarrow\Gamma\longrightarrow A^{\vee}\longrightarrow1
\]
then we have that~$s(\phi_1)\cdots s(\phi_k) = \alpha(\phi_1,\ldots,\phi_k)s(\phi_1\phi_2\cdots \phi_k)$ where the map~\hbox{$\phi\mapsto s(\phi)$} is a set-theoretical section of the projection~$\Gamma\to A^{\vee}$.

We have the formula
\[
\frac{z(\bar{g})(\phi)z(\bar{g})(\psi) }{ z(\bar{g})(\phi\psi) }
=
\frac{ \alpha(\phi,\psi)}{\alpha(g(\phi),g(\psi))}.
\]
By using it repeatedly, we get
\[
a_2(\phi) =\frac{\alpha(\phi,g(\phi),\ldots, g^{k-1}(\phi))}{\alpha(g(\phi),\ldots, g^{k-1}(\phi), \phi)}.
\]
We have also used here the fact that~$z(\bar{g})(\phi g(\phi)\cdots g^{k-1}(\phi))=1$.
This is because the homomorphism~$\phi g(\phi)\cdots g^{k-1}(\phi)$ is~$g$--invariant, and by Lemma~\ref{lem:z} we can assume that the value of~$z(\bar{g})$ on it is 1.

For every~$j=1,\ldots,k-1$ we can also write
\[
a_2(\phi)=z(\bar{g})(g^{j}(\phi)\cdots z(\bar{g})(g^{j+k-1}(\phi))).
\]
By doing the same calculation again, we get
\[
a_2(\phi)=\frac{\alpha(g^j(\phi),\ldots, g^{j+k-1}(\phi))}{\alpha(g^{j+1}(\phi),\ldots, g^{j+k}(\phi))}
\]
for every~$j$. 
We multiply all these expression for~$a_2(\phi)$. They cancel each other, and we are left with~\hbox{$a_2(\phi)^k=1$}.
This finishes the proof.
\end{proof}


\section{Examples, remarks, and a question}\label{sec:examples}

The following examples illustrate the subtlety of the situation.

\begin{example}\label{ex:2}
Let~$\rmD_8$ again denote the dihedral group of order~$8$. We will describe a (non-symmetric) monoidal autoequivalence of~$\Rep(\rmD_8)$ which does not preserve~$\Psi^2$. We will use the following presentation:
\[
\rmD_8=\langle\, x,y,q\,|\,x^2=y^2=q^2=1, [x,y]=1, [q,y]=1, [q,x]=y\,\rangle.
\]
The subgroup~$A=\langle\,x,y\,\rangle$ is a normal subgroup isomorphic to the Klein group. Let us denote a dual pair of generators of~$A^{\vee}$ by~$\nu,\mu$. 
The group~$\rmH^2(A^{\vee}\,;\,\frK^{\times})$ has order~$2$. A generator is given by~$\alpha(\nu^a\mu^b,\nu^c\mu^d) =(-1)^{bc}$.
 This cohomology class must be~$G/A$--invariant.
 We choose~$z(q)(\nu^a\mu^b) = \rmi^b$. We then get that~$z$ satisfies the condition of Lemma~\ref{lem:z}.
 The resulting~$2$--cocycle~$b$ is given by~$b(q,q)=y$. Notice that~$b$ is a trivial cocycle since in~$G$ the equation~$(qx)^2 = y$ holds and therefore the function~$G\to G_b$ given by sending~$x^iy^jq^k$ to~$x^{i+k}y^jq^k$ is an isomorphism of groups.
 However, the resulting monoidal equivalence~\hbox{$\Rep(\frK G)\to \Rep(\frK G_b)\to \Rep(\frK G)$} is not symmetric and it does not preserve the second Adams operation.
 Indeed, the group~$G$ has four 1-dimensional irreducible representations~$V_{00},V_{01},V_{10},V_{11}$ and a unique two dimensional irreducible representation~$W$. The element~$x^iy^jq^k$ acts on~$V_{ab}$ by the scalar~$(-1)^{ai+kb}$.
A direct calculation shows that
\[
\Psi^2(W) = V_{00}+ V_{10}+V_{01}-V_{11}
\]
while~$F(W)=W$ and
\[
F(\Psi^2(W)) = V_{00}-V_{10}+V_{01}+V_{11}
\]
is different from~$\Psi^2(F(W))$.
\end{example}


\begin{example}\label{ex:Adams_but_not_symmetric}
There are examples of monoidal equivalences which preserve all Adams operations, but which are not symmetric. 
Indeed, there are Etingof--Gelaki equivalences that preserve all Adams operations without being symmetric. 
To see this, consider the case where the quotient~$Q$ is trivial but the subgroup~$A$ is not. 
The cocycle~$b$ is then of course trivial, and the functor~$F_b$ preserves trivially all the Adams operations. 
However, since the abelian subgroup~$A$ is not trivial,~$\alpha$ is not trivial, and the functor~$F_b$ is not symmetric. The smallest such example is the Klein~$4$--group~\hbox{$\rmC_2\times\rmC_2$}. This group has four~$1$--dimensional irreducible representations, which we shall denote by~$V_{ij}$ for~$i,j\in\{0,1\}$.
On the category~$\Rep(\frK(\rmC_2\times\rmC_2))$ we have two different symmetries. The canonical symmetry~\hbox{$V_{ij}\ot V_{kl}\cong V_{kl}\ot V_{ij}$} is given by
\[
x\ot y\longmapsto y\ot x,
\]
and the second symmetry is given by 
\[
x\ot y\longmapsto(-1)^{il+jk}y\ot x.
\]
However, with respect to both symmetries, the Adams operation~$\Psi^k$ is the identity for every odd~$k$, and~$\Psi^k(V_{ij})=V_{00}$ for every even~$k$.
\end{example}

\begin{remark} 
In all examples of monoidal equivalences~\hbox{$F\colon\Rep(G)\to\Rep(G')$} that preserve all Adams operations that we have seen so far, the groups~$G$ and~$G'$ are isomorphic. If~$F$ is symmetric, then this is clear: an isomorphism is induced by the functor~$F$, because the groups~$G$ and~$G'$ are isomorphic to the automorphism groups of the fiber functors. In the Examples~\ref{ex:Adams_but_not_symmetric}, we have~$G\cong G'$ by construction.
\end{remark}

The preceding remark may be taken as evidence for an affirmative answer to the following. 

\begin{question}
Suppose that~$G$ and~$G'$ are two isocategorical groups, and assume that the induced isomorphism~$\rmR(G)\cong\rmR(G')$ of representation rings arising from a monoidal equivalence~\hbox{$\Rep(G)\to\Rep(G')$} of representation categories preserves~(not only the odd Adams operations but also) the second Adams operation~$\Psi^2$. Do the groups~$G$ and~$G'$ have to be isomorphic? 
\end{question}

The affine symplectic group $G=V\rtimes\mathrm{Sp}(V)$, for $V=\bbF_2^n$, and its non-split partner $G'$ are worth studying in this respect~\cite{Fischer, Davydov, Etingof+Gelaki}. 

Here is more that may be considered in support of an affirmative answer:

\begin{remark} 
Let~$G$ be a group, and let~$V$ and~$W$ be two~$G$--representations. We will use the notation of Section~\ref{isocats}. If~$G_b$ is a group such that~$G$ and~$G_b$ are isocategorical we have two corresponding~$G_b$--representations~$V_b$ and~$W_b$. A slight variation of the results of Section~\ref{sec:FS} allows us to deduce the following: for every permutation~$\sigma$ in the symmetric group~$\frS_n$, we have~\hbox{$\chi_U(\sigma)=\tr(\,\sigma\,|\,U\,)=\tr(\,\sigma\,|\,U_b\,)=\chi_{U_b}(\sigma)$} when we consider the action of~$\sigma$ on the spaces~\hbox{$U=\Hom_{G}(V,W^{\ot n})$} and~\hbox{$U_b=\Hom_{G_b}(V_b,W_b^{\ot n})$}. Since the permutation groups~$\frS_n$ are finite and the ground field is of characteristic zero, this already means that~$U$ and~$U_b$ are actually isomorphic as~$\frS_n$--representations. 
\end{remark}


\section*{Acknowledgment}

Both authors were supported by the Danish National Research Foundation through the Centre for Symmetry and Deformation~(DNRF92). 
The first author was also supported by the Research Training Group 1670 ``Mathematics Inspired by String Theory and Quantum Field Theory.'' 
We thank Alexei Davydov, Lars Hesselholt, Victor Ostrik, and Bj\"orn Schuster for discussions, and the referee for their valuable comments on the exposition.



\vfill

\parbox{\linewidth}{%
Ehud Meir\\
Institute of Mathematics\\
University of Aberdeen\\
Fraser Noble Building\\
Aberdeen AB24 3UE\\
UNITED KINGDOM\\
\href{mailto:meirehud@gmail.com}{meirehud@gmail.com}\\
\phantom{}\\
Markus Szymik\\
Department of Mathematical Sciences\\
NTNU Norwegian University of Science and Technology\\
7491 Trondheim\\
NORWAY\\
\href{mailto:markus.szymik@ntnu.no}{markus.szymik@ntnu.no}}

\end{document}